\documentclass[12pt]{article}
\usepackage{graphicx}
\usepackage{multirow}
\usepackage{colortbl}
\usepackage{float}
\usepackage{array}
\usepackage{amsmath}
\usepackage{amsthm}
\usepackage{amssymb}
\newtheorem{thm}{Theorem}[section]

\begin{document}
\title{{\bf Riemann Zeta Function Expressed as the Difference of Two
Symmetrized Factorials Whose Zeros All Have Real Part of $1/2$}}
\author{Wusheng Zhu \\
}
\date{July 17, 2012}
\maketitle
\begin{abstract}
In this paper, some new results are reported for the study of
Riemann zeta function $\zeta(s)$ in the critical strip
$0\!<\!Re(s)\!<\!1$, such as $\zeta(s)$ expressed in a generalized
Euler product only involving prime numbers. Particularly, some new
absolutely convergent series representations of $\zeta(s)$ based
on binomial expansion are presented. The crucial progress is to
find that $\zeta(s)$ can be expressed as a linear combination of
polynomials of infinite degree, whose consequences are shown in
several aspects: (i) numerically it provides a scenario to
construct very fast convergent algorithm to calculate $\zeta(s)$;
(ii) interestingly it shows that Lagrange interpolation using
infinite number of integer Euler zeta functions reproduces the
exact complex $\zeta(s)$; (iii) surprisingly it demonstrates that
alternating Riemann zeta function (or other entire functions
removing the pole of zeta function) is admissible to Melzak
combinatorial transform for polynomials. Applying the functional
symmetry on $\zeta(s)$ in the form of Melzak transform induces
$\zeta(s)$ being written as the difference of two symmetrized
factorials whose zeros are proved to all have real part of $1/2$.
Furthermore, the two symmetrized factorials are proved to have
interlacing between the two sequences of the imaginary part of
their zeros on upper (or lower) half plane, which ensures the
difference of the two symmetrized factorials [proportional to
$\zeta(s)$] attaining the same feature of zeros with real part of
$1/2$ to endorse Riemann hypothesis.
\end{abstract}
\pagebreak

\section{Introduction}
\quad\quad Riemann zeta function $\zeta(s)$ is originally defined
by analytic continuation via contour integral to extend Euler zeta
function from the domain $Re(s)>1$ to the whole complex
domain$^{\cite{R:1}}$. If restricted on $Re(s)>0$, $\zeta(s)$ has
many equivalent representations either by series summation or
real variable integration.

A popular series definition of $\zeta(s)$ on $Re(s)\!>\!0$ is via
Dirichlet eta function $\eta(s)$:
\begin{eqnarray}
\zeta(s)=\dfrac{1}{1-2^{1-s}}\eta(s)=\dfrac{1}{1-2^{1-s}}\displaystyle{\sum_{n=1}^{\infty}}\dfrac{(-1)^{n+1}}{n^s}=\dfrac{1}{1-2^{1-s}}\left(
1-\dfrac{1}{2^s}+\dfrac{1}{3^s}-\dfrac{1}{4^s}+\cdots\right)~~\label{eq1}
\end{eqnarray}
where the series is conditionally convergent on $Re(s)\!>\!0$, and
$\zeta(s)$ has a pole at $s\!=\!1$ [$\eta(s)$ is an entire
function with $\eta(1)\!=\!\ln 2]$. For $s\!=\!1\!+\!i 2n\pi/\ln 2
$ $(n\!\ne\! 0)$ satisfying $1\!-\!2^{1-s}\!=\!0$ were proved to
be zeros of $\eta(s)$ but not zeros of $\zeta(s)$ $^{\cite{R:2}}$.
All other zeros of $\eta(s)$ are the same as those of $\zeta(s)$.
Riemann hypothesis is equivalent to state that any $s$ on
$0\!<\!Re(s)\!<\!1$ which satisfies $\eta(s)\!=\!0$ must have
$Re(s)\!=\!1/2$ and $Im(s)\!\ne\! 0$.

A well-known (Fourier) integral representation of $\zeta(s)$ on
$Re(s)>0$ is
\begin{eqnarray}
\zeta(s)=\dfrac{1}{(1-2^{1-s})\Gamma(s)}\int_0^\infty
\dfrac{x^{s-1}}{e^x+1}dx=\dfrac{1}{(1-2^{1-s})\Gamma(s)}\int_{-\infty}^\infty
\dfrac{e^{-[{\rm Re}(s)] t}}{e^{e^{-t}}+1} e^{-i [{\rm Im}(s)] t}
dt\label{eq2}
\end{eqnarray}
There exist many integrals for $\zeta(s)$. Two examples are shown
below (proofs omitted):
\begin{eqnarray}
\int_0^\infty x^{n-2+s}
\dfrac{d^n}{dx^n}\!\left(\dfrac{x}{e^x\!-\!1}\!\right) dx = (-1)^n
(s-1)\Gamma(n-\!1\!+\!s)\zeta(s)~~~~~(n\!\in\! \mathbb{N},\;Re(s)\!>0)~~\label{eq2.3}\\
\int_0^\infty x^{n-2+s}
\dfrac{d^n}{dx^n}\ln\!\left(\dfrac{x}{1\!-\!e^{-x}}\!\right)dx =
(-1)^{n} \Gamma(n-1\!+\!s)\zeta(s)~~~~~(n\!\in\!
\mathbb{N}_0,\;0\!<\!Re(s)\!<\!1)~~\label{eq2.4}
\end{eqnarray}

We only consider $Re(s)>0$ due to the fact that $\zeta(s)$ has a
functional symmetry:
\begin{eqnarray}
\zeta(1-s)=2^{1-s}\pi^{-s}cos(\frac{\pi
s}{2})\Gamma(s)\zeta(s)\label{eq3}
\end{eqnarray}
which can explore $\zeta(s)$ on $Re(s)<0$ by its symmetric
counterpart on $Re(s)>0$. Furthermore, on $Re(s)>1$, Equation
(\ref{eq1}) becomes Euler zeta function which has no zeros proved
by Euler product expansion. Thus only the critical strip
$0<Re(s)<1$ needs to focus where nontrivial zeros may exist.
\begin{thm}
A real positive $s$ can not be a zero of Riemann zeta function
such that $\zeta(s)\ne 0$ if $Im(s)=0$ and $Re(s)>0$.
\label{thm1.1}
\end{thm}
\begin{proof}
From Eq.~(\ref{eq2}) on $Re(s)>0$, if $\zeta(s)=0$, it must
require that
\begin{eqnarray}
\int_0^\infty \dfrac{x^{Re(s)-1}}{e^x+1} \cos[Im (s)\ln x]
dx=\cos[Im (s)\ln x_1]\int_0^\infty \dfrac{x^{Re(s)-1}}{e^x+1} dx
=0 \label{eq4}
\end{eqnarray}
where mean value theorem for integral is applied, and $x_1\in (0,
\infty)$. As the latter integral is positive, it must have
$\cos[Im (s)\ln x_1]=0$ that can not hold with $Im(s)= 0$.
\end{proof}
Theorem \ref{thm1.1} is simple and known but very important result
as it implies that any polynomial expression of $\zeta(s)$ will
not have any real zero on $Re(s)>0$.

\begin{thm}
The functional symmetry in Eq.~(\ref{eq3}) holds on
$0\!<\!Re(s)\!<\!1$.\label{thm1.2}
\end{thm}
\begin{proof}
From another integral for $\zeta(s)$ on $0\!<\!Re(s)\!<\!1$ and
series expansion of ${\rm csch}(x)$:
\begin{eqnarray}
(1-2^{-s})\Gamma(s)\zeta(s)=\int_0^\infty x^{s-1}
\left[\dfrac{1}{2}{\rm
csch}(x)-\dfrac{1}{2x}\right]dx=\int_0^\infty x^{s}
\sum_{k=1}^{\infty} \dfrac{(-1)^k}{x^2+\pi^2 k^2} dx\label{eq5}
\end{eqnarray}
we have the following identity:
\begin{eqnarray}
(1\!-\!2^{-s})\Gamma(s)\zeta(s)&\!\!=\!\!&\int_0^\infty\!
\sum_{k=1}^{\infty} (-)^k
k^{s-1}\dfrac{\left(\dfrac{x}{k}\right)^{s}}{\left(\dfrac{x}{k}\right)^2\!+\!\pi^2}
d\left(\dfrac{x}{k}\right)=(2^{s}\!-\!1)\,\zeta(1\!-\!s)
\int_0^\infty\!\!\!\! \dfrac{x^{s}}{x^2\!+\!\pi^2} dx\nonumber\\
&\!\!=\!\!&(2^{s}-1)\zeta(1-s) \dfrac{\pi^s}{2cos(\frac{\pi
s}{2})}~~~~\label{eq6}
\end{eqnarray}
which is a simple proof of the functional symmetry in
Eq.~(\ref{eq3}) on $0\!<\!Re(s)\!<\!1$.
\end{proof}
Theorem \ref{thm1.2}, which is well-known since Riemann, is the
key to answer Riemann hypothesis. The proof here is very simple,
though only in critical strip.

\begin{thm} $\zeta(s)$ on $Re(s)>0$ can be expressed into a generalized Euler product:
\begin{eqnarray}
\zeta(s)\!=\!\left(\prod_{j=1}^k \dfrac{1}{1-p_j^{-s}}\right)
\left[1\!+\!\lim_{m\to\infty}\!\!\displaystyle{ \sum_{ \substack{2j-1=p_{k+1} \\
2j-1 \;\nmid \;3,5,7,\cdots,\,p_k }}^{p_{k+m}}\!\!\! (2j-1)^{-s}}
-\dfrac{p_{k+m}^{1-s}}{1\!-\!s}\prod_{j=1}^{k}(1\!-\!p_j^{-1})\right]
~~~~\label{eq11}
\end{eqnarray}
where $\{k,m\}\in{\mathbb N}$, and $(2j-1)$ includes the prime
numbers larger than $p_k$ and all their possible products less
than a sufficiently large prime number $p_{k+m}$.
 \label{thm1.4}
\end{thm}
\begin{proof}
The summation-integral difference for $n\to\infty$ related to
$\zeta(s)$ is
\begin{eqnarray}
1+\dfrac{1}{3^s}+\dfrac{1}{5^s}+\dfrac{1}{7^s}+\cdots+\dfrac{1}{(2n-1)^s}\!-\!\!\int_1^n
\!\!\dfrac{1}{(2x\!-\!1)^s}dx
=(1\!-\!2^{-s})\zeta(s)+\dfrac{1}{2(1\!-\!s)}\label{eq12}
\end{eqnarray}
Multiplying $3^{-s}$ on Eq.~(\ref{eq12}) and subtracting from
Eq.~(\ref{eq12}) yields
\begin{eqnarray}
1\!+\!\dfrac{1}{5^s}\!+\!\cdots\!+\!\dfrac{1}{(2n\!-\!1)^s}
\!-\!(1\!-\!\frac{1}{3})\!\int_1^n \!\!\!\dfrac{1}{(2x-1)^s}dx
=(1\!-\!3^{-s})(1\!-\!2^{-s})\zeta(s)\!+\!\dfrac{1\!-\!\frac{1}{3}}{2(1\!-\!s)}\label{eq13}
\end{eqnarray}
Repeating the procedure till the $k$th prime number obtains a
generalized Euler product:
\begin{eqnarray}
1+\displaystyle{\sum_{ \substack{j=(1+p_{k+1})/2
\\ 2j-1 \;\nmid \;3,5,7,\cdots,\;p_k }}^{n} (2j-1)^{-s}}
-\dfrac{(2n-1)^{1-s}}{1-s}\displaystyle{\prod_{j=1}^{k}(1-p_j^{-1})}=\zeta(s)\displaystyle{\prod_{j=1}^{k}(1-p_j^{-s})}\label{eq14}
\end{eqnarray}
An infinite large prime number $p_{k+m}$ replacing $2n\!-\!1$
turns Eq.~(\ref{eq14}) to Eq.~(\ref{eq11}).
\end{proof}
On $Re(s)>1$, the summation and product in the square brackets in
Eq.~(\ref{eq11}) are towards zero when $k\!\to\!\infty$ which
gives regular Euler product. On $0\!<\!Re(s)\!<\!1$, if $k$ is
chosen finite, terms in the square brackets in Eq.~(\ref{eq11})
converge after cancellation of divergent components; if
$k\!\to\!\infty$ is chosen, and $\zeta(s)\!\ne\! 0$, the
right-hand-side of Eq.~(\ref{eq11}) will be in the form of
$0\cdot\infty$ in which L'Hopital's rule applies. If
$\zeta(s)\!=\!\zeta(1\!-\!s)\!=\!0$, considered Eq.~(\ref{eq14}),
 the following equation holds for $n\!\to\!\infty$:
\begin{eqnarray}
\dfrac{\Bigg|1+\displaystyle{\sum_{ \substack{j=(1+p_{k+1})/2
\\ 2j-1 \;\nmid \;3,5,7,\cdots,\;p_k }}^{n} (2j-1)^{-s}}\Bigg|}
{\Bigg|1+\displaystyle{\sum_{ \substack{j=(1+p_{k+1})/2
\\ 2j-1 \;\nmid \;3,5,7,\cdots,\;p_k }}^{n} (2j-1)^{s-1}}\Bigg|}
=\dfrac{|s|}{|1-s|}(2n-1)^{1-2Re(s)}\label{eq15}
\end{eqnarray}
which implies an equivalent Riemann hypothesis as: {\it On
$0\!<\!Re(s)\!\le\! 1/2$, if $\zeta(s)\!=\!0$, the norm ratio in
the left-hand-side of Eq.~(\ref{eq15}) converges to $1$ when
$n\!\to\!\infty$ and $k\in{\mathbb N}$.}

\section{Binomial Series Expansion of $\zeta(s)$}
\quad\quad Convergent series expansions of $\zeta(s)$ are very
limited. The following is such an example based on absolutely
convergent binomial series:
\begin{eqnarray}
\eta(s)= 1+\sum_{n=1}^{\infty}\dfrac{(n-s)!}{n!(1-s)!}
\sum_{j=2}^\infty (-1)^{j-1}
\dfrac{n}{j^2}\left(1-\dfrac{1}{j}\right)^{n-1}
~~~~~(Re(s)>0)~~~~\label{eq16}
\end{eqnarray}
Equation (\ref{eq16}) is arrived by the series definition of
$\eta(s)$ in Eq.~(\ref{eq1}):
\begin{eqnarray}
\eta(s)&=&1+\dfrac{1}{1-s}\sum_{j=2}^\infty (-1)^{j-1}
\left(\dfrac{\partial}{\partial x_j}
\left[1-\left(1-\dfrac{1}{x_j}\right)
\right]^{s-1}\right)\bigg|_{x_j=j}\label{eq17}
\end{eqnarray}
followed by a binomial series expansion on $(1\!-\!X)^{s-1}$ with
$X\!=\!1\!-1/j$ that converges absolutely for finite $j$, while
may diverge for infinite $j$ due to $|X|\!\to\! 1$. Fortunately,
due to the factor $n/j^2$ arisen from the derivative, the series
in Eq.~(\ref{eq16}) will converge.
\begin{thm}
The binomial series expansion of $\eta(s)$ in Eq.~(\ref{eq16})
converges absolutely on $Re(s)\!>\!0$.\label{thm2.1}
\end{thm}
\begin{proof}
Comparison test of the series in Eq.~(\ref{eq16}) with Euler zeta
function will be done as follows. For a sufficient large $n$, the
absolute value of $n$th term in Eq.~(\ref{eq16}) is
\begin{eqnarray}
\left|\dfrac{(n-s)!}{n!(1\!-\!s)!} \right|\left|\sum_{j=2}^\infty
\dfrac{n(-1)^{j-1}}{j^2}\left(1\!-\!\dfrac{1}{j}\right)^{n-1}\right|\!=O\left(\dfrac{1}{n^{Re(s)}}\right)\left|\sum_{j=2}^\infty
\dfrac{n(-1)^{j-1}}{j^2}\left(1\!-\!\dfrac{1}{j}\right)^{n-1}\right|~~\label{eq18}
\end{eqnarray}
where the binomial coefficients asymptotic expansion
\begin{eqnarray}
\binom{n-s}{n}=\dfrac{1}{n^s
(-s)!}\left[1+\binom{s}{2}\frac{1}{n}+O\left(\frac{1}{n^{2}}\right)\right]\label{eq19}
\end{eqnarray}
is applied. For the summation of $j$ in Eq.~(\ref{eq18}), we have
\begin{eqnarray}
&&\!\!\!\!\!\!\!\!\left|\sum_{j=1}^{\infty}\!\!\left[
\dfrac{n}{(2j)^2}\!\!\left(\!1\!-\!\dfrac{1}{2j}\!\right)^{\!\!n-1}\!\!\!\!\!-
\dfrac{n}{(2j\!+\!1)^2}\!\!\left(\!1\!-\!\dfrac{1}{2j\!+\!1}\!\!\right)^{\!\!n-1}
\right]\right|\!<\! \sum_{j=1}^{\infty}\!\left[
\dfrac{n}{(2j)^2}\!-\!
\dfrac{n}{(2j\!+\!1)^2}\right]\!\!\left(\!1\!-\!\dfrac{1}{2j}\!\right)^{\!\!n-1}\nonumber\\
&&\qquad\qquad< \sum_{j=1}^{\infty} \dfrac{n}{(2j)^2
j}e^{-\frac{n-1}{2j}}=O\left(\frac{1}{n}\right)~~~~~\label{eq20}
\end{eqnarray}
where the last step is realized by comparison with the integral
result. Thus for sufficiently large $n$, considered
Eqs.~(\ref{eq20}) and (\ref{eq18}), it has
\begin{eqnarray}
\left|\dfrac{(n\!-\!s)!}{n!(1\!-\!s)!} \!\sum_{j=2}^\infty\!
\dfrac{n(-1)^{j-1}}{j^2}\!\!\left(\!1\!-\!\dfrac{1}{j}\right)^{\!\!n-1}\right|
= O\left(\!\frac{1}{n^{1+Re(s)}}\right)\label{eq22}
\end{eqnarray}
which proves the absolute convergence of Eq.~(\ref{eq16}) on
$Re(s)>0$ as compared to the $n$th term of Euler zeta function
$\zeta(\alpha)$ whose absolute convergence occurs on
$Re(\alpha)>1$.
\end{proof}
In addition, many slightly different convergent binomial series
expansion of $\zeta(s)$ or $\eta(s)$ can be derived from a general
form of
\begin{eqnarray}
(1-2^{1-s})\zeta(s)=\dfrac{\left(-\frac{s-\alpha}{\gamma}\right)!}{\left(q-\frac{s-\alpha}{\gamma}\right)!}\sum_{j=1}^\infty
\dfrac{(-1)^{j-1}}{j^\alpha}\!\!
\left.\left(\dfrac{\partial^{\;q}}{\partial x_j^q}
\!\left[1\!-\!\!\left(1\!-\!\dfrac{1}{x_j^\beta}\right)
\right]^{\frac{\frac{s-\alpha}{\gamma}-q}{\beta}}\right)\right|_{x_j=j^\gamma}\label{eq23}
\end{eqnarray}
where $q\in{\mathbb N},\;\{\alpha,\beta,\gamma\}\in{\mathbb C}$
with certain constraints for convergence. A special case of
Eq.~(\ref{eq23}) on binomial series expansion of $(1-s)\zeta(s)$
 has been studied in literature$^{\cite{R:7}}$. If choosing
$\{\alpha\!=\!0,\;\beta\to 0,\;\gamma\!=\!q\!=\!1\}$, it induces
Taylor series expansion of $\eta(s)$ at $s=1$:
\begin{eqnarray}
\eta(s)&=&\eta(1)+\lim_{\beta \to
0}\sum_{n=2}^{\infty}\dfrac{(-1)^n}{1-s}\binom{\frac{s-1}{\beta}}{n}
\sum_{j=2}^\infty (-1)^{j-1}
\dfrac{n\;\beta}{j^{1+\beta}}\left(1-\dfrac{1}{j^\beta}\right)^{n-1}\nonumber\\
&=&\eta(1)+\lim_{\beta \to
0}\sum_{n=2}^{\infty}\dfrac{(-1)^n}{1-s}\binom{\frac{s-1}{\beta}}{n}
n \beta \sum_{\ell=0}^{n-1} (-1)^\ell \binom{n-1}{\ell}
\sum_{m=1}^{\infty}\dfrac{\eta^{(m)}(1)}{m!}[\beta(\ell+1)]^m\nonumber\\
&=&\eta(1)+\sum_{n=2}^{\infty} (1-s)^{n-1}
\dfrac{\eta^{(n-1)}(1)}{(n-1)!} (-1)^{n-1}\label{eq24}
\end{eqnarray}
There also exist convergent binomial series of $\zeta(s)$ [instead
of $\eta(s)$ in Eq.~(\ref{eq16})] such as
\begin{eqnarray}
\zeta(s)= \dfrac{1}{s-1}+1+\sum_{n=1}^{\infty}
\dfrac{(n-s)!}{n!(1-s)!} \left[-1+ \sum_{j=2}^{\infty}
\dfrac{n}{j^2}\left(1-\dfrac{1}{j}\right)^{n-1}
\right]~~~~(Re(s)>0)~~~~\label{eq16.2}
\end{eqnarray}
But we will focus on the series expansion in Eq.~(\ref{eq16}),
which is equivalent to
\begin{eqnarray}
\eta(s) =\eta(2)+\sum_{n=2}^{\infty}
\dfrac{\chi(n)}{n!}\displaystyle{\prod_{j=2}^{n}}(j-s)~~~~~~~~~~(Re(s)>0)
\label{eq25}
\end{eqnarray}
where $\chi(n)$ $(n\ge 1)$ is defined as
\begin{eqnarray}
\chi(n)=\sum_{j=2}^\infty (-1)^{j-1}
\dfrac{n}{j^2}\left(1-\dfrac{1}{j}\right)^{n-1}\label{eq26}
\end{eqnarray}
Some properties of $\chi(n)$ are $\chi(1)=\eta(2)-1$,\;\;
$\displaystyle{\sum_{n=1}^\infty }\frac{\chi(n)}{n}=\eta(1)-1$,\;
and
\begin{eqnarray}
\chi(\infty)\!=\!\lim_{n\to\infty}\dfrac{1}{2}\sum_{j=1}^{\infty}\!\Bigg[
 \dfrac{\left(1\!-\!\frac{1}{n\left(\frac{2j+1}{n}\right)}\right)^{\!\!n-1}}{\left(\frac{2j+1}{n}\right)^2}
-
\dfrac{\left(1\!-\!\frac{1}{n\left(\frac{2j}{n}\right)}\right)^{\!\!n-1}}{\left(\frac{2j}{n}\right)^2}
\Bigg]\dfrac{2}{n}=0~~\label{eq27}
 \end{eqnarray}
For a finite $n$, $\chi(n)$ is the difference between midpoint and
top-right corner rectangle approaches to the same integral,
approximately $\exp({-1/x})/x^2$. As the sign of the difference
summation in Eq.~(\ref{eq27}) depends on $n$, $\chi(n)$ converges
to zero in oscillation pattern [it is a subtle issue on the zeros
of $\chi(z)$ as analytic continuation of $\chi(n)$].

Applying further binomial expansion on $\chi(n)$ in
Eq.~(\ref{eq26}) will turn Eq.~(\ref{eq16}) into
\begin{eqnarray}
\eta(s)&=&\sum_{n=1}^{\infty}\dfrac{(n-s)!}{(1-s)!}
\sum_{k=0}^{n-1}
\dfrac{(-1)^{k}\eta(k+2)}{k!(n-k-1)!}~~~~~~~(Re(s)>0)\label{eq28}
\end{eqnarray}
where the binomial expansion on $(1-X)^{n-1}$ with $|X|=1/j\le 1$
and $n\in{\mathbb N}$ absolutely converges, so does the series
expansion in Eq.~(\ref{eq28}). It is worth to mention that
numerically Eq.~(\ref{eq16}) is more favorable due to fast
convergence of $\chi(n)$, while Eq.~(\ref{eq28}) suffers a
catastrophic cancellation problem in summation of $k$ due to
alternating binomial coefficients, which relies on high precision
of inputs in order for accurate output.

The crucial step towards revealing the zeros feature of $\zeta(s)$
is to switch the order of the two summations in Eq.~(\ref{eq28})
as follows:
\begin{eqnarray}
\eta(s)&=&\lim_{m\to\infty}\sum_{n=1}^{m+1}
\dfrac{(n-s)!}{(1-s)!}\sum_{k=0}^{n-1}
\dfrac{(-1)^{k}\eta(k+2)}{k!(n-k-1)!}\nonumber\\
&=&\lim_{m\to\infty}\sum_{k=0}^{m}
\dfrac{(-1)^{k}\eta(k+2)}{(1-s)!k!} \sum_{n=k+1}^{m+1}
\dfrac{(n-s)!}{(n-k-1)!}\label{eq28.2}
\end{eqnarray}
which can be further written into
\begin{eqnarray}
\eta(s)&=&\lim_{m\to\infty}\dfrac{(m+2-s)!}{m!(1-s)!}\sum_{k=0}^{m}
\binom{m}{k}
\dfrac{(-1)^k\eta(k+2)}{k+2-s}\nonumber\\
&=&\lim_{m\to\infty}\sum_{k=0}^{m}
\dfrac{(-1)^k\eta(k+2)}{k!(m-k)!}\prod_{\substack{j=0 \\
j\ne k}}^{m}(j+2-s)\label{eq29}
\end{eqnarray}
The validity of switching the two summations in Eq.~(\ref{eq28.2})
and the limit existence of $m\!\to\!\infty$ are ensured by
absolute convergence of the two series. Equation (\ref{eq29}) can
be numerically verified. For example, a small $m\!=\!32$ in
Eq.~(\ref{eq29}) outputs a quite accurate
$\zeta(1/2)\!=\!-1.46034778$. If $\{\alpha\!=\!0,\; \beta\!=\!1,\;
\gamma\!=\!2L,\; q\!=\!J\}$ are chosen in Eq.~(\ref{eq23}), a
general version of Eq.~(\ref{eq29}) can be developed for $\eta(s)$
by using a subset of $\zeta(2k)$:
\begin{eqnarray}
\eta(s)\!
=\!\lim_{m\to\infty}\!\dfrac{(m\!+\!1\!+\!J\!-\!\frac{s}{2L})!}{m!(J-\frac{s}{2L})!}\sum_{k=0}^{m}
\binom{m}{k} \dfrac{(-1)^k
\eta[2L(k\!+\!J\!+\!1)]}{k\!+\!J\!+\!1\!-\!\frac{s}{2L}}~~~~~(J\in{\mathbb
N_0}, L\in{\mathbb N})~~~~\label{eq30}
\end{eqnarray}

\subsection{$\zeta(s)$ calculated by fast convergent scheme.}
\quad\quad Equation (\ref{eq30}) indicates that $\zeta(s)$ can be
calculated by a subset of $\zeta(2k)$ with large $k$ resulting in
fast convergence. A numerical scheme to calculate $\zeta(s)$ with
fast convergence can be constructed by using the following
identity for $\zeta(s)$ with $k\in {\mathbb N},$ and $Re(s)>0$:
\begin{eqnarray}
[1-(2k-1)^{s-1}](2^{s}-1)\zeta(s)= \sum_{n=2}^{\infty} \Bigg[1-
\dfrac{\displaystyle{\sum_{j=1}^{2k-1}}
(2j-1)^n}{(2k-1)^{n+1}}\Bigg] \dfrac{\Gamma(n+s)}{2^n
n!\Gamma(s)}\zeta(n+s)\label{eq31}
\end{eqnarray}
\begin{proof}
Starting from the integral representation of $\zeta(s)$
\begin{eqnarray}
\int_0^\infty\dfrac{ e^{\frac{x}{2}}\!-1}{\!e^{x}-1} x^{s-1} dx
=(2^s\!-\!2)\Gamma(s)\zeta(s)~~~~~(Re(s)>0)~~~~~\label{eq32}
\end{eqnarray}
we use geometric series summation and integral variable scaling in
Eq.~(\ref{eq32}) to obtain
\begin{eqnarray}
[1\!-\!(2k-1)^{s-1}](2^s\!-\!1)\Gamma(s)\zeta(s)\!=\!\!
\int_0^\infty\!\dfrac{e^{\frac{x}{2}}\!-\!\frac{1}{2k-1}\displaystyle{\sum_{j=1}^{2k-1}}
e^{\frac{(2j-1)x}{4k-2}}}{e^{x}-1}x^{s-1}dx~~~~~(Re(s)>0)~~~~~\label{eq33}
\end{eqnarray}
where $k\in {\mathbb N}$. Then Eq.~(\ref{eq31}) is arrived by
taking power series expansion of exponential functions on
numerator in Eq.~(\ref{eq33}) and integrating each term of the
power series.
\end{proof}
In Eq.~(\ref{eq31}), $k=2$ corresponds to the numerical scheme for
calculation of $\zeta(s)$ as
\begin{eqnarray}
(1-3^{1-s})(1-2^{-s})\zeta(s)=1-\dfrac{2}{3^s}+\dfrac{1}{5^s}+\dfrac{1}{7^s}-\dfrac{2}{9^s}+\dfrac{1}{11^s}\cdots
~~~~~ \label{eq34}
\end{eqnarray}
which is arrived by multiplying $3^{1-s}$ on Eq.~(\ref{eq12}) and
subtracting from Eq.~(\ref{eq12}). The series in Eq.~(\ref{eq34})
can be written into the form of Eq.~(\ref{eq31}) by binomial
expansion:
\begin{eqnarray}
\sum_{k=1}^{\infty}\left[\dfrac{1}{(6k\!-\!5)^s}-\dfrac{2}{(6k\!-\!3)^s}+\dfrac{1}{(6k\!-\!1)^s}\right]
\!\!=\! \sum_{n=2}^{\infty} \left[
\dfrac{\Gamma(n+s)}{\Gamma(s)\;n!}\dfrac{(1\!-\!2\cdot
3^n\!+\!5^n)}{6^{n+s}} \zeta(n\!+\!s)\right]\label{eq35}
\end{eqnarray}
where $1\!-\!2\cdot 3^n\!+\!5^n=0$ for $n=0,1$ leads to the
convergence order of $O(n^{-1-2Re(s)})$ roughly as fast as
$\zeta(2)$. In Eq.~(\ref{eq31}), $k=3$ corresponds to
\begin{eqnarray}
(1-5^{1-s})(1-2^{-s})\zeta(s)=1+\dfrac{1}{3^s}-\dfrac{4}{5^s}+\dfrac{1}{7^s}+\dfrac{1}{9^s}+\dfrac{1}{11^s}+\dfrac{1}{13^s}-\dfrac{4}{15^s}+\dfrac{1}{17^s}+\cdots~~~~\label{eq36}
\end{eqnarray}
Linear combination of Eqs.~(\ref{eq34}), (\ref{eq36}) and so on
can generate series to calculate $\zeta(s)$ converging as fast as
$\zeta(N)$, provided that proper combination coefficients (solved
from a set of linear equations) let all terms of $n<N$ vanish in
the summation of $n$ when combining $k$ in Eq.~(\ref{eq31}). For
example, the following combined series with the specific
coefficients converges as fast as $\zeta(6)$ on $Re(s)>0$:
\begin{eqnarray}
&&\!\!\!\!\!\dfrac{6^s}{3}\!\left(\!1\!-\!\dfrac{2}{3^s}\!+\!\dfrac{1}{5^s}+\cdots\right)\!-\dfrac{10^{s+2}}{756}\!\left(\!1\!+\!\dfrac{1}{3^s}\!-\!\frac{4}{5^s}\!+\cdots
\right)\!\!+\!\dfrac{18^{s+1}}{70}\!\left(\!
1\!+\!\dfrac{1}{3^s}\!+\!\dfrac{1}{5^s}\!+\!\dfrac{1}{7^s}\!-\!\frac{8}{9^s}\!+\cdots\right)\nonumber\\
&&\quad=\left[(3^{s-1}-1)-\dfrac{625}{189}(5^{s-1}-1)+\dfrac{81}{35}(9^{s-1}-1)\right](2^{s}-1)\Gamma(s)\zeta(s)
~~~~~~~\label{eq37}
\end{eqnarray}
It shows that Eqs.~(\ref{eq1}), (\ref{eq34}), and (\ref{eq37})
need total $10^6$, $10^4$, and $85$ terms [ $15,25,45$ terms
chosen for the alternating series in Eq.~(\ref{eq37}) from left to
right, respectively] to achieve $0.37471336-0.27518432i$, $
0.36010325-0.26624621i$, and $0.36010259-0.26624619$,
respectively, for calculation of $\zeta(0.2\!+\!2i)$ whose value
is $0.36010259-0.26624620i$.

\subsection{$\eta(s)$ reproduced by
Lagrange interpolation on a set of infinite number of integer eta
functions.}
\quad\quad Explicitly Eq.~(\ref{eq29}) can be
reformed into
\begin{eqnarray}
\eta(s)=\lim_{m\to\infty}\sum_{k=0}^{m} \prod_{\substack{j=0 \\
j\ne k}}^{m}[s\!-\!(j\!+\!2)]
\dfrac{(-1)^{k+m}\eta(k\!+\!2)}{k!(m-k)!} =\sum_{k=0}^{\infty}
\eta(k\!+\!2) \prod_{\substack{j=0
\\ j\ne k}}^{\infty}\dfrac{s-(j+2)}{(k\!+\!2)\!-\!(j\!+\!2)}~~~~~\label{eq38}
\end{eqnarray}
which is an exact Lagrange interpolation formula on infinite
number of integers $2,3,4,\cdots$. The interpolation also can be
done on a subset of integers based on Eq.~(\ref{eq30}). There
exist infinite number of good table of nodes for convergent
interpolation of $\eta(s)$. And some fast convergent iteration
methods$^{\cite{R:11}}$ developed for Lagrange interpolation type
equation can be applied to numerically find all the roots of
$\eta(s)=0$ simultaneously.

\subsection{$\eta(s)$ is admissible to Melzak transform for polynomials.}
\quad\quad Melzak transform is inherited from combinatorial
identities and finite difference theory for polynomials of finite
degree. A basic Melzak transform is defined as:
\begin{thm}
If $f(x)$ is a polynomial of degree $m$, the following transform
holds
\begin{eqnarray}
f(x-y)=y\binom{m-y}{m}\sum_{k=0}^m(-1)^k
\binom{m}{k}\dfrac{f(x-k)}{y-k}~~~~~~~(y\ne
0,1,\cdots,m)~~~~\label{eq39}
\end{eqnarray}
for $x,y\in{\mathbb C}$. Choosing $x=0$ yields a special case:
\begin{eqnarray}
f(-y)=y\binom{m-y}{m}\sum_{k=0}^m(-1)^k
\binom{m}{k}\dfrac{f(-k)}{y-k}~~~~~~(y\ne
0,1,\cdots,m)~~~~\label{eq40}\end{eqnarray}\label{thm2.2}\vspace{-30pt}
\end{thm}
The proof is in literature$^{\cite{R:10}}$. Equation (\ref{eq29})
shows that $\eta(s)$ is admissible to Melzak transform by
replacing $y\!=\!s-2,\;f(-k)\!=\!\eta(k+2),$ and
$f(-y)\!=\!\eta(s)$ in Eq.~(\ref{eq40}). Like analytic
continuation, Melzak transform can be used to extend a function
from integer domain into complex domain (e.g., defining complex
index Bernoulli numbers).

Admissible condition to Melzak transform for $m\!\to\!\infty$ can
be used to characterize an entire function which behaves like a
pseudo-finite degree polynomial. If $f_m(-k)$ is the truncated
polynomial of $f(-k)$ by cutting off all degrees above $m$, then
the requirement for $f(-k)$ being admissible to Melzak transform
in Eq.~(\ref{eq40}) for $m\!\to\!\infty$ is
\begin{eqnarray}
\lim_{m\to\infty} y\binom{m-y}{m}\sum_{k=0}^m(-1)^k
\binom{m}{k}\dfrac{[f(-k)-f_m(-k)]}{y-k}=0~~~~~~(y\ne
0,1,\cdots,m)~~~~~\label{eq41}
\end{eqnarray}
For instance, exponential decay and sinc damping are admissible,
which satisfy both $ \displaystyle{\lim_{k\to\infty}f(-k)=0}$ and
$\displaystyle{\lim_{m\to\infty}f^{(m)}(0)/m=0}$,
 while cosine and sine are not admissible.

\section{$\zeta(s)$ expanded by zeros of symmetrized factorials}
\quad\quad From Eq.~(\ref{eq29}), $\eta(s)=0$ is equivalent to
\begin{eqnarray}
\lim_{m\to\infty}\sum_{k=0}^{m} \binom{m}{k}
\dfrac{(-1)^k\eta(k+2)}{k+2-s}=0\label{eq42}
\end{eqnarray}
which unfortunately can not be solved analytically. To unveil the
feature of $\eta(s)\!=\!0$ on $0\!<\!Re(s)\!<\!1$, we turn to
apply various factorization by zeros on $\eta(s)$ in
Eq.~(\ref{eq29}).

\subsection{$\eta(s)$ as the summation of symmetrized
factorials of all even degrees whose zeros all have real part of
$1/2$.} \quad\quad The functional symmetry of $\zeta(s)$ or
$\eta(s)$ is the core reason to cause the special feature of
$\zeta(s)=0$. We found that symmetrized factorials arisen from the
functional symmetry have all zeros with real part of $1/2$.
\begin{thm}
All roots of the following symmetrized factorial polynomial
equation
\begin{eqnarray}
\displaystyle{\prod_{j=1}^{n}\left(a_j-1+s\right)+\prod_{j=1}^{n}\left(a_j-s\right)}=0~~~~~~(n>1)\label{eq43}
\end{eqnarray}
have real part of $1/2$ and nonzero imaginary part if all real
$(a_j\!-1/2)$ do not vanish simultaneously and the nonzero
$(a_j\!-1/2)$ have the same sign.\label{thm3.1}\end{thm}
\begin{proof}
Change the variable to be $x=s-1/2$, Equation~(\ref{eq43}) becomes
\begin{eqnarray}
\prod_{j=1}^n \left(a_j-\dfrac{1}{2}+x\right)= -\prod_{j=1}^n
\left(a_j-\dfrac{1}{2}-x\right)\label{eq44}
\end{eqnarray}
If any root $x$ has nonzero real part as $x=\delta+i\tau$, then
Eq.~(\ref{eq44}) requires
\begin{eqnarray}
\frac{\displaystyle{\prod_{j=1}^n}\bigg|a_j-\dfrac{1}{2}+\delta+i\tau\bigg|^2}
{\displaystyle{\prod_{j=1}^n}\bigg|a_j-\dfrac{1}{2}-\delta-i\tau\bigg|^2}
=\frac{\displaystyle{\prod_{j=1}^n} \left[\left(
a_j-\dfrac{1}{2}+\delta\right)^2+\tau^2\right]}
{\displaystyle{\prod_{j=1}^n}\left[\left(a_j-\dfrac{1}{2}-\delta\right)^2+\tau^2\right]}=1\label{eq45}
\end{eqnarray}
However, when all $(a_j-1/2)$ have the same sign, if $\delta\ne
0$, the norm ratio in Eq.~(\ref{eq45}) will always be greater or
smaller than $1$. Equation (\ref{eq45}) holds only if $\delta=0$
or $x=i\tau$. Thus all $2{\lfloor n/2\rfloor}$ roots of
Eq.~(\ref{eq43}) have real part of $1/2$ and nonzero imaginary
part [ Eq.~(\ref{eq44}) can not hold for real $s\!=\!1/2$ (i.e.,
$x\!=\!0$) when at least one $(a_j\!-1/2)\ne 0$].
\end{proof}

The anti-symmetrized version of Eq.~(\ref{eq43}) (the difference
of the two factorials) can also be proved to have all roots with
real part of $1/2$ (and $s=1/2$ is a root too). In this paper we
only focus on the symmetrized version.

Applying the functional symmetry of Eq.~(\ref{eq3}) on $\eta(s)$
in Eq.~(\ref{eq25}), we have
\begin{eqnarray}
\eta(s)+\eta(1-s) =\zeta(2)+\sum_{n=2}^{\infty}
\dfrac{\chi(n)}{n!}\left[\displaystyle{\prod_{j=0}^{n-2}}(j+2-s)+\displaystyle{\prod_{j=0}^{n-2}}(j+1+s)\right]
~~~~~\label{eq46}
\end{eqnarray}
which is valid in the critical strip because both $\eta(s)$ and
$\eta(1-s)$ expanded via alternating series in Eq.~(\ref{eq1})
are valid on $0<Re(s)<1$. Considered Theorem $\ref{thm3.1}$,
Equation (\ref{eq46}) can be factorized by zeros of symmetrized
factorials of all even degrees:
\begin{eqnarray}
&&\!\!\!\!\!\!\!\!\eta(s)+\eta(1-s)\nonumber\\
&&\!\!\!\!\!=\zeta(2)\!+\!\dfrac{3\chi(2)}{2!}+\sum_{n=3}^{\infty}
\!\!\left(\![1\!-\!(-1)^n]\!+\!\dfrac{n^2\!-\!1}{2}[1\!+\!(-1)^n]\right)\dfrac{\chi(n)}{n!}
\prod_{j=1}^{{\lfloor
\frac{n-1}{2}\rfloor}}\!\!\left[\left(s\!-\!\frac{1}{2}\right)^2\!+\!\nu_{n,j}^2\right]
~~~~~~\label{eq47}
\end{eqnarray}
where $\{1/2\pm i\nu_{n,j}\}$ are the complex conjugated zeros of
the symmetrized factorial of degree $2{\lfloor (n-1)/2\rfloor}$
in the square brackets in Eq.~(\ref{eq46}). The first few
$\{\nu_{n,j}\}$ are
\begin{eqnarray}
\{\nu_{3,j}\}\!=\pm\dfrac{1}{2}\sqrt{15},\;\;\{\nu_{4,j}\}\!=\pm\dfrac{1}{2}\sqrt{7},
\;\;\{\nu_{5,j}\}\!=\pm\dfrac{1}{2}\sqrt{\!103\!\pm
8\sqrt{151}}~~\label{eq48}
\end{eqnarray}
with $j=1,2,\cdots,2{\lfloor (n-1)/2\rfloor}$. If $\eta(s)$ is
expanded by Taylor series in Eq.~(\ref{eq24}), since the zeros of
$(1\!-\!s)^n\!+\!s^n$ are solvable, the factorization by zeros
becomes
\begin{eqnarray}
&&\!\!\!\!\!\!\!\!\!\!\!\eta(s)+\eta(1-s)
\nonumber\\
&&\!\!\!\!\!=2\eta(1)\!-\!\eta^{(1)}(1)\!+\!\sum_{n=2}^{\infty}
\dfrac{(-1)^n\eta^{(n)}(1)}{n!}\prod_{k=1}^{\lfloor
\frac{n}{2}\rfloor}\left\{\left(s\!-\!\frac{1}{2}\right)^2\!\!+\!\dfrac{1}{4}\!\left[\tan\left(\frac{(2k-1)\pi}{2n}\right)\right]^2\right\}
~~~~~\label{eq49}
\end{eqnarray}
A trivial case is Taylor expansion at $s=1/2$ as factorization by
a single repeated zero.

\subsection{$\eta(s)$ as the summation of symmetrized
factorials of the same infinite degree whose zeros all have real
part of $1/2$.} \quad\quad If the functional symmetry is applied
on $\eta(s)$ in Eq.~(\ref{eq29}), it has
\begin{eqnarray}
\eta(s)+\eta(1-s)= \lim_{m\to\infty}\sum_{k=0}^{m}
\dfrac{(-1)^k\eta(k+2)}{k!(m-k)!}\Bigg[\prod_{\substack{j=0 \\
j\ne k}}^{m}(j+2-s)+\prod_{\substack{j=0 \\ j\ne k}
}^{m}(j+1+s)\Bigg] ~~~~\label{eq50}
\end{eqnarray}
Considered Theorem $\ref{thm3.1}$, $\eta(s)$ in Eq.~(\ref{eq50})
can be factorized by zeros of the symmetrized factorials of
infinite degree as
\begin{eqnarray}
\eta(s)+\eta(1-s) &=& \lim_{m\to\infty}\sum_{k=0}^{m}
\dfrac{2(-1)^k\eta(k+2)}{k!(m-k)!}
\prod_{j=1}^{\frac{m}{2}}\left[\left(s-\frac{1}{2}\right)^2+\tau_{k,j}^2\right]
~~~~\label{eq51}
\end{eqnarray}
where $\{1/2 \pm i\tau_{k,j}\}$ are the complex conjugated zeros
of each symmetrized factorial in the square brackets of
Eq.~(\ref{eq50}), and an even $m$ is chosen for convenience in
this paper.

\subsection{$\eta(s)$ as the difference of two symmetrized factorials of the same infinite degree whose zeros all have real
part of $1/2$.} \quad\quad Equation (\ref{eq50}) is a linear
combination of the symmetrized factorials of the same degree. To
deal with the case of linear combination, we have the following
result:
\begin{thm}
The equation below is a linear combination of symmetrized
factorials:
\begin{eqnarray}
\sum_{k=0}^{m} c_k \Bigg[\prod_{\substack{j=0 \\
j\ne k}}^m (a_j-1+s)+\prod_{\substack{j=0 \\
j\ne k}}^m (a_j-s)\Bigg]=0~~~~~~~~(m>1)~~~~\label{eq52}
\end{eqnarray}
with all coefficients $c_k$ of the same sign. Then all roots of
Eq.~(\ref{eq52}) must have real part of $1/2$ and nonzero
imaginary part if all real ($a_j-1/2$) do not vanish
simultaneously, and the nonzero ($a_j-1/2$) have the same
sign.\label{thm3.2}
\end{thm}
\begin{proof}
Equation (\ref{eq52}) requires
\begin{eqnarray}
\dfrac{\displaystyle{\sum_{k=0}^{m}c_k\prod_{\substack{j=0 \\
j\ne k}}^{m}(a_j-s)}} {\displaystyle{\sum_{k=0}^{m} c_k
\prod_{\substack{j=0 \\
j\ne k}}^{m}(a_j-1+s)}}=
\dfrac{\displaystyle{\prod_{j=0}^{m}(a_j-s)} \sum_{k=0}^{m}
\dfrac{c_k}{a_k-s}} {\displaystyle{\prod_{j=0}^{m}(a_j-1+s)}
\sum_{k=0}^{m} \dfrac{c_k}{a_k-1+s}}=-1~~~~\label{eq53}
\end{eqnarray}
where the norm ratio should be $1$:
\begin{eqnarray}
\dfrac{\displaystyle{\Bigg|\prod_{j=0}^{m}(a_j-s)\Bigg|^2}\Bigg|
\sum_{k=0}^{m} \dfrac{c_k}{a_k-s}\Bigg|^2}
{\displaystyle{\Bigg|\prod_{j=0}^{m}(a_j-1+s)\Bigg|^2}
\Bigg|\sum_{k=0}^{m}
\dfrac{c_k}{a_k-1+s}\Bigg|^2}=1~~~~\label{eq54}
\end{eqnarray}
Assume that there exists a root $s=1/2+\delta+i\tau$ for
Eq.~(\ref{eq52}). We can define the norms
\begin{eqnarray}
r_j\equiv|a_j-s|=\sqrt{(a_j-\frac{1}{2}-\delta)^2+\tau^2}~~~~\label{eq55}\\
R_j\equiv|a_j-1+s|=\sqrt{(a_j-\frac{1}{2}+\delta)^2+\tau^2}~~~~\label{eq56}
\end{eqnarray}
for $j=0,1,\cdots,m$. Then Eq.~(\ref{eq54}) becomes
\begin{eqnarray}
1&\!\!=\!\!&\dfrac{\displaystyle{\left(\prod_{j=0}^{m}
r_j^2\right) \left(\left[\sum_{k=0}^{m}
\dfrac{c_k}{r_k^2}\left(a_k-\frac{1}{2}-\delta\right)\right]^2+\tau^2\left[\sum_{k=0}^{m}
\dfrac{c_k}{r_k^2}\right]^2\right)}
}{\displaystyle{\left(\prod_{j=0}^{m}R_j^2\right)
\left(\left[\sum_{k=0}^{m}
\dfrac{c_k}{R_k^2}\left(a_k-\frac{1}{2}+\delta\right)\right]^2+\tau^2
\left[\sum_{k=0}^{m}
\dfrac{c_k}{R_k^2}\right]^2\right)}}\nonumber\\
&\!\!=\!\!&\dfrac{\displaystyle{\sum_{k=0}^{m}
c_k^2\prod_{\substack{j=0 \\
j\ne k}}^{m} r_j^2 + \sum_{k=0}^{m-1}\sum_{\ell=k+1}^{m} 2c_k
c_\ell\left
[\left(a_k-\frac{1}{2}-\delta\right)\left(a_\ell-\frac{1}{2}-\delta\right)+\tau^2\right]\prod_{\substack{j=0 \\
j\ne k\,,\,\ne \ell}}^{m} r_j^2}} {\displaystyle{\sum_{k=0}^{m}
c_k^2\prod_{\substack{j=0 \\
j\ne k}}^{m} R_j^2 + \sum_{k=0}^{m-1}\sum_{\ell=k+1}^{m} 2c_k
c_\ell
\left[\left(a_k-\frac{1}{2}+\delta\right)\left(a_\ell-\frac{1}{2}+\delta\right)+\tau^2\right]\prod_{\substack{j=0 \\
j\ne k\,,\,\ne \ell}}^{m} R_j^2}} ~~~~~~~\label{eq57}
\end{eqnarray}
Suppose that all real ($a_j-1/2$) do not vanish simultaneously,
and the nonzero ($a_j-1/2$) have the same sign. If $\delta$ has
the same sign as the nonzero ($a_j-1/2$), it has $r_j\!<\!R_j$ for
all $j$, then Eq.~(\ref{eq57}) can not hold for all $c_k$ of the
same sign because the numerator is always smaller than the
denominator. If $\delta$ and the nonzero ($a_j-1/2$) have opposite
sign, then the numerator is always greater than the denominator.
Thus if $s=1/2+\delta+i\tau$ is a root of Eq.~(\ref{eq52}), it
must have $\delta\!=\!0$. And for real $s\!=\!1/2$ (i.e.,
$\tau\!=\!0$), Equation (\ref{eq53}) can not hold when at least
one $(a_j-1/2)\ne 0$.
\end{proof}

From Theorem $\ref{thm3.2}$, we have the following factorization
by roots of Eq.~(\ref{eq52}):
\begin{eqnarray}
\sum_{k=0}^{m} c_k \Bigg[\prod_{\substack{j=0 \\
j\ne k}}^m (a_j-s)+\prod_{\substack{j=0 \\
j\ne k}}^m (a_j-1+s)\Bigg]= 2 \left(\sum_{k=0}^{m} c_k
\right)\prod_{j=1}^{\frac{m}{2}}\left[\left(s-\frac{1}{2}\right)^2+
T_{j}^2\right]~~~~\label{eq58}
\end{eqnarray}
where $\{1/2\pm i T_j\}$ are the corresponding roots. On the other
hand, from Theorem $\ref{thm3.1}$, we have another factorization
by zeros of each symmetrized factorial:
\begin{eqnarray}
\sum_{k=0}^{m} c_k \Bigg[\prod_{\substack{j=0 \\
j\ne k}}^m (a_j-s)+\prod_{\substack{j=0 \\
j\ne k}}^m (a_j-1+s)\Bigg]= \sum_{k=0}^{m} c_k\left(
2\prod_{j=1}^{\frac{m}{2}}\left[\left(s-\frac{1}{2}\right)^2+\tau_{k,j}^2\right]\right)~~~~\label{eq59}
\end{eqnarray}
where $\{1/2\pm i \tau_{k,j}\}$ are the zeros of the $k$th
symmetrized factorial. Comparing Eq.~(\ref{eq59}) to
Eq.~(\ref{eq58}) for $s=1/2$ reveals one correlation among the
imaginary parts of all zeros:
\begin{eqnarray}
\prod_{j=1}^{\frac{m}{2}}
T_{j}^2=\dfrac{\displaystyle{\sum_{k=0}^{m} c_k
\prod_{j=1}^{\frac{m}{2}}\tau_{k,j}^2}}{\displaystyle{\sum_{k=0}^{m}
c_k}}~~~~\label{eq60}
\end{eqnarray}

Specifically, in Eq.~(\ref{eq50}), if the same sign coefficients
are grouped separately, Theorem $\ref{thm3.2}$ can be applied on
Eq.~(\ref{eq50}) to factorize by zeros of two combined
polynomials:
\begin{eqnarray}
\eta(s)+\eta(1\!-\!s)&\!\!=\!\!& \lim_{m\to\infty} 2
\left(\sum_{k=0}^{\frac{m}{2}}
\dfrac{\eta(2k+2)}{(2k)!(m-2k)!}\right)
\prod_{j=1}^{\frac{m}{2}}\left[\left(s-\frac{1}{2}\right)^2+\Theta_{j}^2\right]\nonumber\\
&&-\lim_{m\to\infty} 2 \left(\sum_{k=0}^{\frac{m}{2}-1}
\dfrac{\eta(2k+3)}{(2k+1)!(m-2k-1)!}\right)
\prod_{j=1}^{\frac{m}{2}}\left[\left(s-\frac{1}{2}\right)^2+\Phi_{j}^2\right]~~~~~\label{eq61}
\end{eqnarray}
where $\{1/2\pm i \Theta_j\}$ and $\{1/2\pm i \Phi_j\}$ are the
zeros of the combined polynomials of even and odd $k$ terms,
respectively. Equation (\ref{eq61}) is arrived by first applying
the functional symmetry on Eq.~(\ref{eq29}), and then combining
the symmetrized factorials with coefficients of the same sign in
Eq.~(\ref{eq50}). The order can be switched. Before applying the
functional symmetry, we can first combine the factorials with
coefficients of the same sign in Eq.~(\ref{eq29}), or
alternatively add an auxiliary function to convert all the
combination coefficients in Eq.~(\ref{eq29}) to be positive and
then subtract it [which avoids to deal with even and odd $k$
terms]. For example, choosing
$f(x)={\binom{2x+2m}{2m}}/{\binom{x+m}{m}}$ for Melzak transform
in Eq.~(\ref{eq39}), then for $m\ge 1$, $y\ne 0,1,\cdots,m $, and
$x=0$, we have
\begin{eqnarray}
y\binom{m\!-\!y}{m}\sum_{k=0}^m
\binom{2k}{k}\binom{2m\!-\!2k}{m\!-\!k}\dfrac{1}{y-k}=\dfrac{\displaystyle{\binom{2m}{m}\binom{2m\!-\!2y}{2m}}}{\displaystyle
{\binom{m\!-\!y}{m}}}=\dfrac{2^{2m}}{m!}\prod_{j=0}^{m-1}
(j\!+\!\frac{1}{2}-y)~~~~\label{eq62}
\end{eqnarray}
Substituting $y=s-2$ in Eq.~(\ref{eq62}) reduces to
\begin{eqnarray}
\dfrac{\pi}{[(\frac{m-1}{2})!]^2}\prod_{j=0}^{m-1}
(j+\frac{5}{2}-s)=\sum_{k=0}^m
\dfrac{2^{-2m}\pi}{[(\frac{m-1}{2})!]^2} \binom{2k}{k}\binom{2m-2k}{m-k} \prod_{\substack{j=0 \\
j\ne k}}^m (j+2-s)\label{eq63}
\end{eqnarray}
which is a special case of the following identity from Melzak
transform:
\begin{eqnarray}
\dfrac{\pi}{\sin(\pi\beta)} \prod_{j=0}^{m-1}
(j\!+\!\gamma\!+\!1\!-\!\beta\!-\!s)=\sum_{k=0}^m \dfrac{(m\!-\!k\!-\!\beta)!(k\!+\!\beta\!-\!1)!}{k!(m-k)!}  \prod_{\substack{j=0 \\
j\ne k}}^m (j+\gamma-s)~~~\label{eq64}
\end{eqnarray}
where $\beta \notin \mathbb{Z}$ and $\{\beta,\gamma\}\in{\mathbb
C}$.

For $0\le k\le m$, the summation coefficients in Eq.~(\ref{eq63})
are always positive:
\begin{eqnarray}
\dfrac{2^{-2m}\pi}{[(\frac{m-1}{2})!]^2}
\binom{2k}{k}\binom{2m-2k}{m-k}=\dfrac{1}{[(\frac{m-1}{2})!]^2}
\dfrac{(k-\frac{1}{2})!(m-k-\frac{1}{2})!}{k!(m-k)!}\ge
\dfrac{1}{k!(m-k)!}\label{eq65}
\end{eqnarray}
where we considered the fact that for $0\le k\le m$,
$\binom{m}{k}$ has peak value at $\binom{m}{m/2}$.

Taking $m\!\to\!\infty$ in Eq.~(\ref{eq63}) and adding with
Eq.~(\ref{eq29}) obtains
\begin{eqnarray}
\eta(s)&\!\!\!=&\!\!\!\!\lim_{m\to\infty}\sum_{k=0}^m
\!\dfrac{\!(-1)^k\eta(k\!+\!2)\!+\!
\dfrac{(k\!-\!\frac{1}{2})!(m\!-\!k\!-\!\frac{1}{2})!}{[(\frac{m-1}{2})!]^2}}{k!(m-k)!} \! \prod_{\substack{j=0 \\
j\ne k}}^m
(j\!+\!2\!-\!s)\!-\!\dfrac{\pi}{\left[(\frac{m-1}{2})!\right]^2}\!
\prod_{j=0}^{m-1} (j\!+\!\frac{5}{2}\!-\!s)\nonumber\\
&\!\!\equiv\!\!&\lim_{m\to\infty}\sum_{k=0}^m \epsilon_k  \prod_{\substack{j=0 \\
j\ne k}}^m (j+2-s)-\dfrac{\pi}{\left[(\frac{m-1}{2})!\right]^2}
\prod_{j=0}^{m-1} (j+\frac{5}{2}-s)\label{eq66}
\end{eqnarray}
where all $\epsilon_k>0$ due to Eq.~(\ref{eq65}) and
$\eta(k+2)<1$. The sum of the coefficients is
\begin{eqnarray}
\sum_{k=0}^m \epsilon_k &\!\!\!\equiv\!\!\!&\sum_{k=0}^m
\!\dfrac{(-1)^k\eta(k\!+\!2)\!+\!
\dfrac{(k\!-\!\frac{1}{2})!(m\!-\!k\!-\!\frac{1}{2})!}{[(\frac{m-1}{2})!]^2}}{k!(m-k)!}
=\dfrac{\chi(m\!+\!1)}{(m\!+\!1)!}\!+\!\dfrac{\pi}{\left[(\frac{m-1}{2})!\right]^2}~~~~~~
\label{eq67}
\end{eqnarray}
where $\chi(m+1)$ is defined in Eq.~(\ref{eq26}).

Equation (\ref{eq66}) can be numerically verified. For example,
choosing a small $m\!=\!20$ in Eq.~(\ref{eq66}) outputs
$0.70082616\!+\!0.43532002i$ for $\eta(0.2\!+\!2i)$ whose value is
$0.70077353\!+\!0.43513124i$. For combination with all
$\epsilon_k\!>\!0$ in Eq.~(\ref{eq66}), we have the following
result:

\begin{thm}
Assume that $f(x)$ is a polynomial of degree $m$ with positive
leading coefficient, and have all real distinct zeros of
$\{a_1,a_2,\cdots,a_m\}$. Polynomial $g(x)$ of degree $m\!-\!1$
is obtained by linear combination as
\begin{eqnarray}
g(x)=c_1\dfrac{f(x)}{x-a_1}+\cdots+c_m\dfrac{f(x)}{x-a_m}\label{eq68}
\end{eqnarray}
Then $g(x)$ has $m\!-\!1$ real zeros $\{b_1,b_2,\cdots,b_{m-1}\}$
which interlace with $m$ real zeros of $f(x)$ as
$a_1\!<\!b_1\!<\!a_2\!<\cdots<\!a_{m-1}\!<\!b_{m-1}\!<\!a_m$ if
and only if all $c_i$ are positive.\label{thm3.3}
\end{thm}

The proof can be found elsewhere$^{\cite{R:14}}$. Applying Theorem
$\ref{thm3.3}$ on Eq.~(\ref{eq66}) yields
\begin{eqnarray}
\eta(s)&\!\!=\!\!&\lim_{m\to\infty}\left[\left(\!
\dfrac{\chi(m\!+\!1)}{(m\!+\!1)!}\!+\!\dfrac{\pi}{\left[(\frac{m-1}{2})!\right]^2}\!\right)
\!\prod_{j=0}^{m-1}
(j\!+\!2\!+d_j\!-\!s)-\dfrac{\pi}{\left[(\frac{m-1}{2})!\right]^2}
\prod_{j=0}^{m-1} (j\!+\!\frac{5}{2}\!-\!s)\right]\nonumber\\
&\!\!\equiv\!\!& \lim_{m\to\infty}
\left[f_d(s)-f_{h}(s)\right]\label{eq69}
\end{eqnarray}
where $0\!<\!d_j\!<\!1$ is required by roots interlacing from
Theorem $\ref{thm3.3}$.

Theorem $\ref{thm1.1}$ denies any real zero for $\eta(s)$ on
$Re(s)>0$. Thus there can not exist any interlacing segment for
three or more zeros between the zeros
$\{2\!+\!d_0,3\!+\!d_1,\cdots,m\!+\!1\!+\!d_{m-1}\}$ of $f_d(x)$
and the zeros $\{2.5,3.5,\cdots,m\!+\!1.5\}$ of $f_h(s)$,
otherwise the combined polynomial will have at least one real
zero. For example, if only one zero $j\!+\!2\!+\!d_{j}$ of
$f_d(s)$ is in-between two consecutive zeros
$\{j\!+\!2.5,j\!+\!3.5\}$ of $f_h(s)$ to form a three-zero
interlacing, then at two ends where $f_h(s)=0$, the combined
polynomial $f_d(s)-f_h(s)$ will have opposite sign, the same as
$f_d(s)$. This means that at least one time sign change occurs
in-between (causing a real zero) for the combined polynomial.
Considered both $0<d_j<1$ and nonexistence of zeros interlacing
segments, the only possibility to arrange the two sequences of
the zeros of $f_d(s)$ and $f_h(s)$ to ensure the combined
polynomial $\eta(s)$ having no real zeros on real positive $s$
will be
\begin{eqnarray}
2.5<2\!+\!d_0<3\!+\!d_1<3.5<4.5<\cdots<m\!+\!d_{m-2}<m\!+\!1\!+\!d_{m-1}<m\!+\!1.5\label{eq70}
\end{eqnarray}
where the smallest zero is from $f_{h}(s)$, since it has $
f_{h}(s)+1/2<f_d(s)<f_{h}(s)+1$ due to $1/2<\eta(s)<1$ for all
real $s>0$. Equation (\ref{eq70}) also can be expressed as
\begin{eqnarray}
0<d_{2\ell+1}<\frac{1}{2}<d_{2\ell}<1\;\;\;\;\;\;(\ell=0,1,\cdots,\frac{m}{2})\label{eq71}
\end{eqnarray}
We numerically verified the validity of Eqs.~(\ref{eq70}) or
(\ref{eq71}). Moreover, here we list a few properties of the two
functions $f_d(s)$ and $f_h(x)$ defined in Eq.~(\ref{eq69}):

(i) Two consecutive zeros of $f_{d}(s)$:
$\{2\ell\!+\!2\!+\!d_{2\ell}, 2\ell\!+\!3\!+\!d_{2\ell+1}\}$ with
$\ell\!\in\! [0,m/2]$ are within the interval of two consecutive
zeros of $f_{h}(s)$: $\{2\ell\!+\!2.5, 2\ell\!+\!3.5\}$ with
$\ell\!\in\! [0,m/2]$.

(ii) On real $s\!>\!0$, $f_{d}(s)$ and $f_{h}(s)$ cross the real
axis $m$ times, and $f_{d}(s)$ is always on top of $f_{h}(s)$ more
than $1/2$ but less than $1$ without crossing each other.

(iii) $f_{h}(s)$ and $f_d(s)$ are monotonic decrease on
$s\in[0,2.5)$, and only have one peak between two consecutive
zeros [as $f_{h}^{(1)}(s)$ and $f_d^{(1)}(s)$ change sign one time
in-between].

The following graph illustrates the displacement of $f_d(s)$ and
$f_h(s)$ on real $s\!>\!0$:
\begin{figure}[!ht]
\centering
\mbox{\includegraphics[scale=0.15]{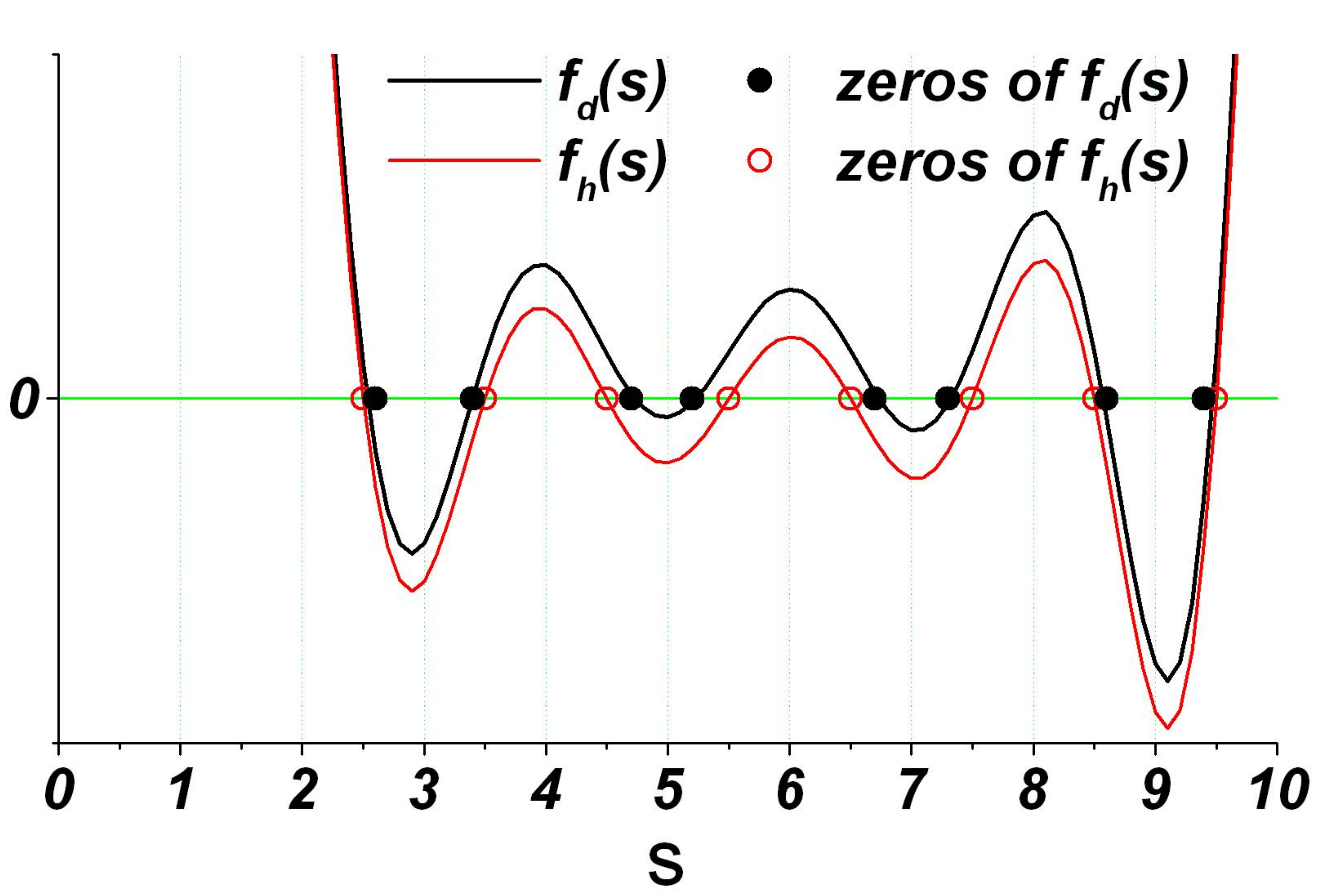}}~~~~~
\caption{\label{fig 1} Schematic graph of factorials $f_d(s)$ and
$f_h(s)$ in Eq.~(\ref{eq69}) on real $s>0$.}
\end{figure}

Considering the functional symmetry of Eq.~(\ref{eq69}), we will
have
\begin{eqnarray}
&&\!\!\!\!\!\!\!\eta(s)+\eta(1-s)=\;\lim_{m\to\infty}\Bigg\{-\;
\dfrac{\pi}{\left[(\frac{m-1}{2})!\right]^2}
\left[\prod_{j=0}^{m-1} (j\!+\!\frac{5}{2}\!-\!s)+
\prod_{j=0}^{m-1}
(j\!+\!\frac{3}{2}\!+\!s)\right]\nonumber\\
&&\qquad+\left(
\dfrac{\chi(m\!+\!1)}{(m\!+\!1)!}\!+\!\dfrac{\pi}{\left[(\frac{m-1}{2})!\right]^2}\right)\!\left[\prod_{j=0
}^{m-1} (j\!+\!2\!+\!d_j\!-\!s)+\prod_{j=0}^{m-1}
(j\!+\!1\!+\!d_j\!+\!s)\right]\Bigg\}~~~~\label{eq72}
\end{eqnarray}
which shows that $\eta(s)$ [proportional to $\eta(s)+\eta(1-s)$]
can be expressed as the difference of two symmetrized factorials
of infinite degree that all zeros of each symmetrized factorial
have real part of $1/2$. In details, applying Theorem
$\ref{thm3.1}$ on Eq.~(\ref{eq72}) obtains
\begin{eqnarray}
\eta(s)+\eta(1-s)&=&\lim_{m\to\infty}
\Bigg\{-\dfrac{2\pi}{\left[(\frac{m-1}{2})!\right]^2}
\prod_{j=1}^{\frac{m}{2}}\left[\left(s-\frac{1}{2}\right)^2+{
\lambda}_{j}^2\right]\nonumber\\
&&\qquad+\left(
\dfrac{2\chi(m\!+\!1)}{(m\!+\!1)!}\!+\!\dfrac{2\pi}{\left[(\frac{m-1}{2})!\right]^2}\right)
\prod_{j=1}^{\frac{m}{2}}\left[\left(s-\frac{1}{2}\right)^2+{
\omega}_{j}^2\right]\Bigg\}~~~~~\label{eq73}
\end{eqnarray}
In Eq.~(\ref{eq73}), $\pm{\lambda}_j$ ($j=1,2,\cdots,m/2$) are $m$
real roots of the following equation
\begin{eqnarray}
F_h(x)&\equiv&
\prod_{j=0}^{m-1} (j+2-ix)+ \prod_{j=0}^{m-1} (j+2+ix)\nonumber\\
&=& 2\cos\left( \sum_{j=0}^{m-1}\arctan\left(\frac{x}{j+2}\right)
\right)\prod_{j=0}^{m-1}
\sqrt{(j+2)^2+x^2}\;\;=0~~~~~~~~~~~~~~~~~~~~~\label{eq74}
\end{eqnarray}
which is arrived by substituting $s\!=\!1/2\!+\!ix$ ($x$ is a real
number!) in Eq.~(\ref{eq72}) as the zeros of the symmetrized
factorials based on Theorem $\ref{thm3.1}$. Similarly
$\pm{\omega}_j$ ( $j=1,2,\cdots,m/2$) in Eq.~(\ref{eq73}) are $m$
real roots of the following equation:
\begin{eqnarray}
F_d(x)&\equiv&\prod_{j=0 }^{m-1}
(j+\frac{3}{2}+d_j-ix)+\prod_{j=0}^{m-1} (j+\frac{3}{2}+d_j+ix)
\nonumber\\
&=&2 \cos\left(
\sum_{j=0}^{m-1}\arctan\left(\frac{x}{j+\frac{3}{2}+d_j}\right)
\right)\prod_{j=0}^{m-1}
\sqrt{\left(j+\frac{3}{2}+d_j\right)^2+x^2}\;\;=0~~~~~~~\label{eq75}
\end{eqnarray}
Below we will analyze the relationship between the roots of
Eqs.~(\ref{eq74}) and (\ref{eq75}).

\subsection{$\eta(s)$ as a single polynomial of infinite degree whose zeros all have real part of
$1/2$.} \quad\quad Since Eqs.~(\ref{eq75}) and (\ref{eq74}) are
polynomials of only even power of $x$ [all odd power of $ix$ in
$F_h(x)$ and $F_d(x)$ vanish otherwise $F_h(x)$ and $F_d(x)$ can
not be real], we can define the new variable as $y=x^2$ so that
$\{\omega_j^2\}$ with $j=1,2,\cdots,m/2$ are the roots of
$F_d(y)=0$ and $\{\lambda_j^2\}$ with $j=1,2,\cdots,m/2$ are the
roots of $F_h(y)=0$. Then we have

\begin{thm}
The $m/2$ distinct real roots $\{\lambda_j^2\}$ of $F_h(y)$ in
Eq.~(\ref{eq74}) (defining $y=x^2$) interlace the $m/2$ distinct
real roots $\{\omega_j^2\}$ of $F_d(y)$ in Eq.~(\ref{eq75})
(defining $y=x^2$) or vice versa as
$\lambda_1^2<\omega_1^2<\lambda_2^2
<\cdots<\lambda_{m/2}^2<\omega_{m/2}^2$ \;or\;
$\omega_1^2<\lambda_1^2<\omega_2^2<
\cdots<\omega_{m/2}^2<\lambda_{m/2}^2$ depending on specific
sequence of $0<d_{2\ell+1}<1/2<d_{2\ell}<1$ in Eq.~(\ref{eq75}).
\label{thm3.4}
\end{thm}
\begin{proof}
First of all, $\{\lambda_j^2\}$ and $\{\omega_j^2\}$ are all real
guaranteed by applying Theorem $\ref{thm3.1}$ on Eq.~(\ref{eq72}).
Finding the roots of $F_h(y)=0$ [or $F_d(y)=0$] is equivalent to
finding the roots of $F_h(x)=0$ [or $F_h(x)=0$] on $x>0$ subject
to a square mapping.

Since $F_d(x)$ and $F_h(x)$ have the same degree and the same sign
of the leading coefficients, if $F_d(x)$ and $F_h(x)$ have sign
interlacing on $x>0$, then their zeros are interlacing on $x>0$.
Without loss of generality, we consider if there is a sign change
of $F_d(x)$ between any two consecutive real positive zeros of
$F_h(x)$. Let the two consecutive zeros of $F_h(x)$ on $x>0$ be
$x_0$ and $x_1$, from Eq.~(\ref{eq74}) we have
\begin{eqnarray}
\cos\left(\sum_{j=0}^{m-1}\arctan\left(\frac{x_0}{j\!+\!2}\right)\right)\!=\!0
\;\;\Longrightarrow\;\;
\sum_{j=0}^{m-1}\arctan\left(\frac{x_0}{j\!+\!2}\right)=\frac{2k\!+\!1}{2}\pi\label{eq76}\\
\cos\left(\sum_{j=0}^{m-1}\arctan\left(\frac{x_1}{j\!+\!2}\right)\right)\!=\!0
\;\;\Longrightarrow\;\;
\sum_{j=0}^{m-1}\arctan\left(\frac{x_1}{j\!+\!2}\right)=\frac{2k\!+\!3}{2}\pi\label{eq77}
\end{eqnarray}
where $k\in{\mathbb{N}_0}$. The sums of the angles in
Eqs.~(\ref{eq76}) and (\ref{eq77})  must have a difference of
$\pi$ because from $x_0$ to $x_1$ is a continuous process and
between $x_0$ and $x_1$, $F_h(x)$ (equivalently the cosine
function) does not change sign (otherwise, they are not
consecutive zeros) so that $x_0$ and $x_1$ corresponds to two
consecutive cosine zeros.

If $F_d(x)$ has sign change on $x_0$ and $x_1$, from
Eq.~(\ref{eq75}) it must satisfy
\begin{eqnarray}
\cos\left(\sum_{j=0}^{m-1}\arctan\left(\frac{x_0}{j\!+\!\frac{3}{2}\!+d_j}\right)\right)
\cos\left(\sum_{j=0}^{m-1}\arctan\left(\frac{x_1}{j\!+\!\frac{3}{2}\!+d_j}\right)\right)<0\label{eq78}
\end{eqnarray}
Considered the inverse tangent formula of sum of angles:
\begin{eqnarray}
\sum_{j=0}^{m-1}\arctan\left(\frac{x}{j\!+\!\frac{3}{2}\!+d_j}\right)
=\sum_{j=0}^{m-1}\arctan\left(\frac{x}{j\!+\!2}\right)-\sum_{j=0}^{m-1}\arctan\left(\frac{x}{j+2+\frac{\rho_j(x)}{d_j-\frac{1}{2}}}\right)\label{eq79}
\end{eqnarray}
with $\rho_j(x)=(j+2)^2+x^2$, Equation (\ref{eq78}) is equivalent
to require
\begin{eqnarray}
\sin\left(\sum_{j=0}^{m-1}\arctan\left(\frac{x_0}{j+2+\frac{\rho_j(x_0)}{d_j-\frac{1}{2}}}\right)\right)
\sin\left(\sum_{j=0}^{m-1}\arctan\left(\frac{x_1}{j+2+\frac{\rho_j(x_1)}{d_j-\frac{1}{2}}}\right)\right)
>0\label{eq80}
\end{eqnarray}
where Eqs.(\ref{eq76}) and (\ref{eq77}) are utilized. Equation
(\ref{eq80}) can be algebraically proved to hold for real positive
$x_0$ and $x_1$ and $d_j$ satisfying Eq.~(\ref{eq71}). But it is
more straightforward to prove Eq.~(\ref{eq80}) by geometric graph
showing that both sums of angles in Eq.~(\ref{eq80}) are less than
$\pi/2$. The following graph shows the complex vectors $a_j+ix$:

\begin{figure}[!ht]
\centering \mbox{\includegraphics[scale=0.5]{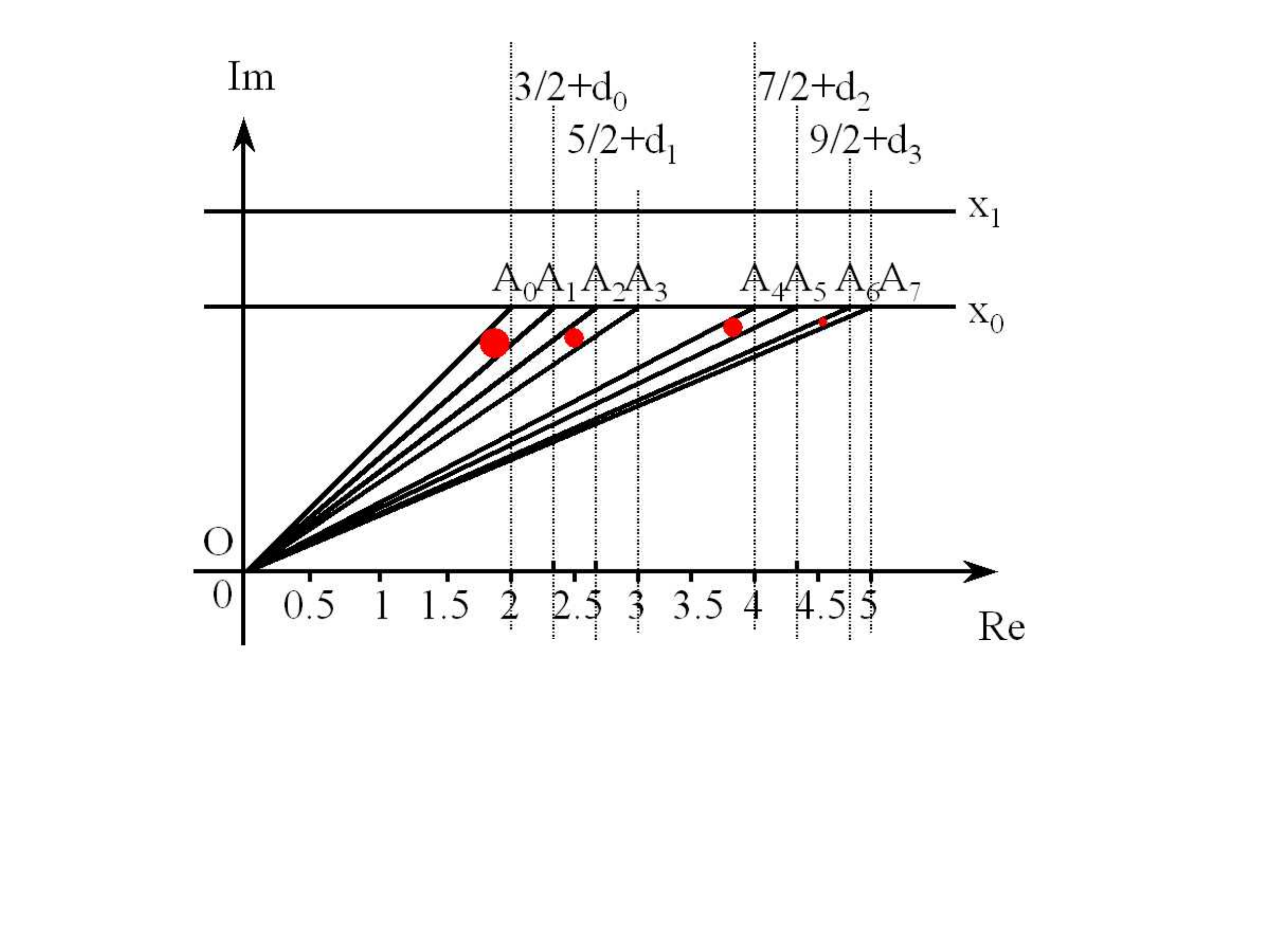}}~~~~~
\caption{\label{fig 2}. Geometric graph of the angles in Eq.
(\ref{eq80}).}
\end{figure}
\noindent in which the angles (containing the balls) are
$\displaystyle{\angle
A_{0}OA_{1}\!=\!\arctan\Big(\dfrac{x_0}{2\!+\!\rho_0(x_0)/(d_0-1/2)}\Big)}$,
$\displaystyle{\angle
A_{2}OA_{3}\!=\!-\arctan\Big(\dfrac{x_0}{3\!+\!\rho_1(x_0)/(d_1-1/2)}\Big)}$
etc. It is obvious that even the sum of absolute angles in
Eq.~(\ref{eq80}) are less than $\pi/2$ so that Eq.~(\ref{eq80})
and then Eq.~(\ref{eq78}) hold, which proves the sign interlacing
between $F_d(x)$ and $F_h(x)$ on $x\!>\!0$. Equivalently, it
concludes that the zeros of $F_h(y)$ and $F_d(y)$ are
interlacing, but the smallest zero can not be determined by a
general sequence of $0<d_{2\ell+1}<1/2<d_{2\ell}<1$ unless
additional information applies.
\end{proof}

In Eq.~(\ref{eq73}), when factorization by zeros between the two
polynomials is done before the limit $m\to\infty$ is taken, then
for a sufficiently large $m$, $\chi(m+1)$ can be chosen as a tiny
negative number for convenience, which is similar to Taylor
expansion of $e^{-x}$ being terminated at a sufficiently large
term with coefficient of either sign while approaching the same
convergent limit. If $\chi(m+1)<0$ is chosen, then applying
Theorem $\ref{thm3.4}$ will lead to an unambiguous zeros
interlacing as
${\lambda}_1^2<\omega_1^2<\lambda_2^2<\cdots<\lambda_{m/2}^2<\omega_{m/2}^2$
due to the requirement of $\eta(1/2)>0$ with $\chi(m+1)<0$ in
Eq.~(\ref{eq73}).

Therefore, changing the variable as $y\!=\!(s\!-\!1/2)^2$,
Equation (\ref{eq73}) indicates that $\eta(s)+\eta(1-s)$ can be
expressed as the difference of two polynomials of $y$:
\begin{eqnarray}
\eta(s)\!+\!\eta(1\!-\!s)\!=\!\lim_{m\to\infty}\!\! \Bigg[\!
\!\left(\!
\dfrac{2\chi(m\!+\!1)}{(m\!+\!1)!}\!+\!\dfrac{2\pi}{\left[(\frac{m-1}{2})!\right]^2}\!\right)
\!\prod_{j=1}^{\frac{m}{2}}\!\left(y\!+\!{\omega}_{j}^2\right)-\dfrac{2\pi}{\left[(\frac{m-1}{2})\!\right]^2}
\prod_{j=1}^{\frac{m}{2}}\!\left(y\!+\!{\lambda}_{j}^2\right)\Bigg]
\label{eq81}
\end{eqnarray}
with interlacing zeros being
$-\omega_{m/2}^2<-\lambda_{m/2}^2<\cdots<
-\lambda_2^2<-\omega_1^2<-\lambda_1^2$ as a result of Theorem
$\ref{thm3.4}$ and $\chi(m+1)<0$.

For general linear combination of two polynomials, the following
result is known:
\begin{thm}
Suppose that $\{a_1,a_2,\cdots,a_n\}$ and $\{b_1,b_2,\cdots,b_n\}$
are all real zeros of $f(x)$ and $g(x)$, respectively, and $g(x)$
interlaces $f(x)$ as $\{b_1<a_1<b_2<\cdots<b_n<a_n\}$. For a
combined polynomial $F(x)$ such that $F(x)=\alpha f(x)+\beta g(x)$
where $\alpha, \beta$ are two real numbers, if $F(x)$ and $g(x)$
have the same degree and have leading coefficients of the same
sign, then $F(x)$ has all real zeros $\{c_1,c_2,\cdots,c_n\}$ and
$f(x)$ interlaces $F(x)$ as $\{a_1<c_1<a_2<\cdots<a_n<c_n\}$,
provided that $\beta<0$.\label{thm3.5}
\end{thm}
A more general version was proved elswhere$^{\cite{R:16}}$.
Equation (\ref{eq81}) can be written into
\begin{eqnarray}
-[\eta(s)\!+\!\eta(1\!-\!s)]&\!\!\!\!=\!\!\!\!&\lim_{m\to\infty}\!
\Bigg[\dfrac{2\pi}{\left[(\frac{m-1}{2})\!\right]^2}
\prod_{j=1}^{\frac{m}{2}}\!\left(y\!+\!{\lambda}_{j}^2\right)
-\!\left(\!
\dfrac{2\chi(m\!+\!1)}{(m\!+\!1)!}\!+\!\dfrac{2\pi}{\left[(\frac{m-1}{2})!\right]^2}\!\right)
\!\prod_{j=1}^{\frac{m}{2}}\!\left(y\!+\!{\omega}_{j}^2\right)
\Bigg]\nonumber\\
&\!\!\!\!\equiv\!\!\!\!& \!\lim_{m\to\infty}\!
\Bigg[\dfrac{2\pi}{\left[(\frac{m-1}{2})!\right]^2}
G_h(y)-\!\left(\!
\dfrac{2\chi(m\!+\!1)}{(m\!+\!1)!}\!+\!\dfrac{2\pi}{\left[(\frac{m-1}{2})!\right]^2}\right)
\!G_d(y)\Bigg]\label{eq82}
\end{eqnarray}
where the two polynomials $G_h(y)$ and $G_d(y)$ have all real
zeros $\{-\lambda_j^2\}$ and $\{-\omega_j^2\}$, respectively, and
their zeros are proved to interlace as
$-\omega_{m/2}^2<-\lambda_{m/2}^2<\cdots<
-\lambda_2^2<-\omega_1^2<-\lambda_1^2$. Moreover, the combined
polynomial $-[\eta(s)+\eta(1-s)]$ and $G_d(y)$ have leading
coefficients of the same sign as the chosen $\chi(m+1)<0$. In
addition, the combination coefficient of $G_d(y)$ is a negative
constant. Therefore Eq.~(\ref{eq82}) satisfies all conditions to
apply Theorem $\ref{thm3.5}$ to yield
\begin{eqnarray}
\eta(s)\!+\!\eta(1\!-\!s)\!=\!\lim_{m\to\infty}\!
\dfrac{2\chi(m\!+\!1)}{(m\!+\!1)!}
\prod_{j=1}^{\frac{m}{2}}\left(y\!+\!{
\Omega_{j}}\right)\!=\!\lim_{m\to\infty}
\!\dfrac{2\chi(m\!+\!1)}{(m\!+\!1)!}
\prod_{j=1}^{\frac{m}{2}}\left[\left(s\!-\!\frac{1}{2}\right)^2\!+\!{
\Omega_{j}}\right]~~\label{eq83}
\end{eqnarray}
where the zeros $\{-\Omega_j\}$ of $\eta(s)+\eta(1-s)$ are all
real and interlacing with the zeros $\{-\lambda_j^2\}$ of
$G_h(y)$ in Eq.~(\ref{eq82}) as
\begin{eqnarray}
-\lambda_{\frac{m}{2}}^2<-\Omega_{\frac{m}{2}}<\cdots<-\lambda_2^2<-\Omega_2<-\lambda_1^2<
-\Omega_1\label{eq84}
\end{eqnarray}
Since $\{\lambda_j^2\}$ are all real positive numbers, Equation
(\ref{eq84}) indicates that $\{\Omega_j\}$ also must be real
positive numbers except for $\Omega_1$ whose sign can not be
determined by Eq.~(\ref{eq84}).

If all $\Omega_j>0$, due to the negative leading coefficient [the
chosen $\chi(m+1)<0$] in Equation~(\ref{eq83}), it will result in
a contradictory $\eta(1/2)<0$. Thus it concludes that all
$\Omega_j>0$ for $j=2,3,\cdots,m/2$ except that $\Omega_1<0$.

However, on the other hand, $\Omega_1<0$ means that
$\eta(s)+\eta(1-s)$ will have two real zeros as
$s=1/2\pm\sqrt{-\Omega_1}$, which seems to be contradictory to
Theorem $\ref{thm1.1}$. In order to avoid this dilemma, it must
have $\Omega_1\le -1/4$ such that the two corresponding "real
zeros" $s\le 0$ and $s\ge 1$ are just out of the convergent
domain (the critical strip) of $\eta(s)+\eta(1-s)$ as
Eq.~(\ref{eq83}) is derived based on Dirichlet series in
Eq.~(\ref{eq1}). Therefore, the factor of $[(s-1/2)^2+\Omega_1]$
still exists, and Theorem $\ref{thm1.1}$ is not violated [within
the convergent domain of $\eta(s)+\eta(1-s)$].

Thus we have the final factorization by zeros for
$\eta(s)+\eta(1-s)$ on $0<Re(s)<1$:
\begin{eqnarray}
\eta(s)+\eta(1-s)&=&\left[(2^{1-s}-2)\pi^{-s} \cos\left(\frac{\pi
s}{2}\right)\Gamma(s)+(1-2^{1-s})\right]\zeta(s)\nonumber\\
&=&\lim_{m\to\infty} \dfrac{2\chi(m\!+\!1)}{(m\!+\!1)!}
\prod_{j=1}^{\frac{m}{2}}\left[\left(s-\frac{1}{2}\right)^2+{
\Omega_{j}}\right]~~~~~~\label{eq85}
\end{eqnarray}
where all $\{\Omega_j\}>0$ except that $\Omega_1\le -\frac{1}{4}$,
and $\{\Omega_j\}$ are determined by Eq.~(\ref{eq72}) as the
difference of two symmetrized factorials. Equation (\ref{eq85})
shows that $\zeta(s)$ is proportional to a single product of
infinite number of quadratic forms $[(s-1/2)^2+\Omega_j]$ with
all $\Omega_j>0$ except that $\Omega_1\le -\frac{1}{4}$, which
immediately endorses Riemann hypothesis in the critical strip.

It is worth to mention that

(i) If $\chi(m+1)>0$ is chosen in Eq.~(\ref{eq73}), it will also
show that $\eta(s)+\eta(1-s)$ can be expressed into a single
polynomial whose all roots have real part of $1/2$. First of all,
following the same procedure will arrive at the same conclusion
that the zeros $\{\lambda_j^2\}$ interlace the zeros
$\{\omega_j^2\}$ or vice versa. In the case of
$\omega_1^2<\lambda_1^2<\omega_2^2<\cdots<\omega_{m/2}^2<\lambda_{m/2}^2$,
Theorem $\ref{thm3.5}$ can be directly applied on Eq.~(\ref{eq81})
to conclude that the zeros $\{-\omega_j^2\}$ interlace the zeros
of $\{-\Omega_j\}$ leading to
\begin{eqnarray}
\eta(s)+\eta(1-s)&=& \lim_{m\to\infty}
\dfrac{2\chi(m\!+\!1)}{(m\!+\!1)!}
\prod_{j=1}^{\frac{m}{2}}\left[\left(s-\frac{1}{2}\right)^2+{
\Omega_{j}}\right]~~~~~~\label{eq83.2}
\end{eqnarray}
with all $\{\Omega_j\}>0$. In the case of
$\lambda_1^2<\omega_1^2<\lambda_2^2<\cdots<\lambda_{m/2}^2<\omega_{m/2}^2$,
defining the variable as $y'\!=\!-(s\!-1/2)^2$, Equation
(\ref{eq73}) becomes
\begin{eqnarray}
\eta(s)\!+\!\eta(1\!-\!s)\!=\!\lim_{m\to\infty}\!\! \Bigg[\!
\!\left(\!
\dfrac{2\chi(m\!+\!1)}{(m\!+\!1)!}\!+\!\dfrac{2\pi}{\left[(\frac{m-1}{2})!\right]^2}\!\right)
\!\prod_{j=1}^{\frac{m}{2}}\!\left(\!{\omega}_{j}^2\!-\!y'\right)-\dfrac{2\pi}{\left[(\frac{m-1}{2})\!\right]^2}
\prod_{j=1}^{\frac{m}{2}}\!\left(\!{\lambda}_{j}^2\!-\!y'\right)\Bigg]
\label{eq81.2}
\end{eqnarray}
Then Theorem $\ref{thm3.5}$ can be applied on Eq.~(\ref{eq81.2})
to conclude that the zeros $\{\omega_j^2\}$ interlace the zeros
$\{\Omega_j\}$ leading to
\begin{eqnarray}
\eta(s)\!+\!\eta(1\!-\!s)\!=\!\lim_{m\to\infty}\!
\dfrac{2\chi(m\!+\!1)}{(m\!+\!1)!}
\prod_{j=1}^{\frac{m}{2}}\left({
\Omega_{j}-y'}\right)\!=\!\lim_{m\to\infty}
\!\dfrac{2\chi(m\!+\!1)}{(m\!+\!1)!}
\prod_{j=1}^{\frac{m}{2}}\left[\left(s\!-\!\frac{1}{2}\right)^2\!+\!{
\Omega_{j}}\right]~~\label{eq83.3}
\end{eqnarray}
with all $\{\Omega_j\}>0$. Thus regardless of the sign of
$\chi(m+1)$, $\eta(s)+\eta(1-s)$ and its approximation with
finite large $m$ will all have the expression of a single
polynomial whose zeros all have real part of $1/2$.

(ii) Similar to Arndt-Gosper formula$^{\cite{R:15}}$, $F_h(x)$ and
$F_d(x)$ in Eqs.~(\ref{eq74}) and (\ref{eq75}) can be written into
explicit forms of polynomial of $x$:
\begin{eqnarray}
F_h(x)&=& \dfrac{2(-1)^{\frac{m}{2}}}{m} \;\sum_{k=0}^{m-1}\;
\prod_{j=0}^{m-1}\left[x-(j+2)\tan\left(\frac{\pi(j-k)}{m}+\frac{\pi}{2m}\right)\right]=0~~~~\label{eq74.2}\\
F_d(x) &=& \dfrac{2(-1)^{\frac{m}{2}}}{m}\; \sum_{k=0}^{m-1}\;
\prod_{j=0}^{m-1}\left[x-(j+\frac{3}{2}+d_j)\tan\left(\frac{\pi(j-k)}{m}+\frac{\pi}{2m}\right)\right]=0~~~~\label{eq75.2}
\end{eqnarray}
for an even $m$. When $m\to\infty$, under midpoint approach,
Equations (\ref{eq74.2}) and (\ref{eq75.2}) become Cauchy
principal integrals:
\begin{eqnarray}
F_h(x)&\!\!=\!\!&\lim_{m\to\infty} 2(-1)^{\frac{m}{2}}\;P.V.\int_{0}^{1} \prod_{j=0}^{m-1}\left[x\!-\!(j\!+\!2)\tan\left(\frac{\pi j}{m}\!-\!\pi y\right)\right]dy=0~~\label{eq74.3}\\
F_d(x)&\!\!=\!\!&\lim_{m\to\infty}
2(-1)^{\frac{m}{2}}\;P.V.\int_{0}^{1}
\prod_{j=0}^{m-1}\left[x\!-\!(j\!+\!\frac{3}{2}\!+\!d_j)\tan\left(\frac{\pi
j}{m}\!-\!\pi y\right)\right]dy=0~~\label{eq75.3}
\end{eqnarray}
where all odd powers of $x$ vanish due to
\begin{eqnarray}
P.V. \int_0^1 \tan\left(\frac{\pi j}{m}-\pi y\right)dy=0
\end{eqnarray}
The interlacing relationship between the zeros $\{\lambda_j^2\}$
of $F_h(y=x^2)$ and $\{\omega_j^2\}$ of $F_d(y=x^2)$ may also be
explored starting from Eqs.~(\ref{eq74.2}) and (\ref{eq75.2}).

(iii) Equation (\ref{eq85}) indicates that all roots of $\zeta(s)$
are a subset (equal to or less than the number) of the total roots
$\{s=1/2\pm i\sqrt{\Omega_j}\}$, since the prefactor of $\zeta(s)$
[i.e., $\kappa(s)\equiv (2^{1-s}-2)\pi^{-s} \cos\left(\pi
s/2\right)\Gamma(s)+(1-2^{1-s})$] might have isolated zeros.
Nevertheless, Equation (\ref{eq85}) implies that no matter where
the roots comes from $\zeta(s)$ or the prefactor $\kappa(s)$, in
the critical strip all their roots must have real part of $1/2$.

(iv) On the other hand, $\zeta(s)$ might have additional zeros
(the trivial ones actually) that are not included in
Eq.~(\ref{eq85}), because the expansion of $\zeta(s)$ in
Eq.~(\ref{eq85}) is derived from Dirichlet series $\eta(s)$ whose
valid domain is $Re(s)>0$. Moreover, to make the functional
symmetry $\eta(1-s)$ also valid via Dirichlet series
representation, the valid domain is restricted in the critical
strip $0<Re(s)<1$. Thus all trivial zeros of $\zeta(s)$ located at
real negative integers do not emerge in Eq.~(\ref{eq85}), and
possible zeros of the prefactor $\kappa(s)$ which are outside the
critical strip also will not show up in Eq.~(\ref{eq85}).

\section{Conclusions}
\quad\quad In this paper, based on absolutely convergent binomial
expansion, alternating Riemann zeta function $\eta(s)$ is found to
be admissible to Melzak transform for infinite degree polynomials.
Specifically, $\eta(s)$ can be expressed as a linear combination
of cyclic polynomials $P_k(s)=\displaystyle{\prod_{j=0,\; j\ne
k}^{m}}(j\!+\!2\!-\!s)$ with $k=0,1,\cdots,m$, which is shown in
Eq.~(\ref{eq29}).

Considered the functional symmetry of Riemann zeta function, the
combined $\eta(s)+\eta(1-s)$, which is proportional to $\eta(s)$,
can be written into a linear combination of symmetrized factorial
polynomials $P_k(s)+P_k(1-s)$, shown in Eq.~(\ref{eq50}). All
roots of $P_k(s)+P_k(1-s)$ for each $k$ have real part of $1/2$.
Moreover, we proved that for a linear combination of
$P_k(s)+P_k(1-s)$ with same sign combination coefficients, all
roots of the combined polynomial will still have real part of
$1/2$.

Riemann hypothesis would be endorsed immediately if the linear
combination coefficients in Eq.~(\ref{eq50}) all had the same
sign. However, the combination coefficients in Eq.~(\ref{eq50})
are alternating between positive and negative, which equivalently
leads to $\eta(s)$ expressed into the difference between two
symmetrized factorials whose roots all have have real part of
$1/2$. Fortunately, we proved that the imaginary parts of the
zeros of the two symmetrized factorials on upper half plane are
interlacing. Based on well-known results about zeros feature of
the combined polynomial from the difference of two interlacing
polynomials, Riemann hypothesis is endorsed to show that the
combined polynomial [proportional to $\zeta(s)$] can be expressed
into a single product of infinite number of quadratic forms
$(s-1/2)^2+\Omega_j$ with all $\Omega_j>0$.


\newpage
\begin{thebibliography}{99}
\bibitem{R:1} \noindent H. M. Edwards, {\it Riemann's Zeta
Function}, Academic Press, New York, 1974.
\bibitem{R:2} \noindent J. Sondow, Zeros of the Alternating Zeta Function on the Line $Re(s)=1$, {\it Amer. Math. Monthly}, 110 (2003), 435-437.
\bibitem{R:7} \noindent L. Baez-Duarte, On Maslanka's representation for the Riemann zeta
function, {\it Int. J. Math. and Math. Sci.}, 2010 (2010), 1-10.
\bibitem{R:10} \noindent Z. A. Melzak, V.D. Gokhale, W. V. Parker, Problem 4458, {\it Amer. Math. Monthly}, 60 (1953), 53-54.
\bibitem{R:11} \noindent X. Zhang, H. Peng, G. Hu, A high order iteration formula for the simultaneous inclusion of polynomial zeros, {\it
App. Math. and Comp.}, 179 (2006), 545-552.
\bibitem{R:14} \noindent S. Fisk, Polynomials, roots, and interlacing,
http://arXiv.org/abs/math/0612833: Lemma 1.20 on page 13.
\bibitem{R:16} \noindent L. Liu, and Y. Wang, A unified approach
to polynomial sequences with only real zeros, {\it Adv. App.
Math.}, 38 (2007), 542-560: Theorem 2.1 on pape 545.
\bibitem{R:15} \noindent
http://mathworld.wolfram.com/InverseTangent.html: Eq.~(56).
\end{thebibliography}
\end{document}